\documentclass[12pt,letterpaper]{amsart}
\usepackage[margin=1.23in,includeheadfoot]{geometry}
\usepackage{amstext}
\usepackage{amsthm}
\usepackage{amsmath}
\usepackage{slashed}
\usepackage{graphicx}
\usepackage{mathrsfs}
\usepackage{txfonts}
\usepackage{amssymb,dsfont}
\usepackage{enumerate}
\usepackage{slashed}
\usepackage{fancyhdr}
\usepackage{aliascnt}
\usepackage{pdfsync}
\usepackage{hyperref}
\usepackage{cite}
\usepackage{color}
\usepackage{setspace}

\allowdisplaybreaks[4]

\newtheorem{theorem}{Theorem}[section]
\newtheorem{corollary}{Corollary}[section]
\newtheorem{lemma}{Lemma}[section]

\theoremstyle{definition}
\newtheorem{definition}{Definition}[section]
\theoremstyle{remark}
\newtheorem{remark}{Remark}[section]
\numberwithin{equation}{section}

\theoremstyle{conjecture}
\newtheorem*{conjecture}{Conjecture}
\theoremstyle{conjecture1}
\newtheorem*{conjecture1}{Yau's Conjecture}
\theoremstyle{conjecture2}
\newtheorem*{conjecture2}{Generalized Yau's Conjecture}

\newcounter{stepnum}

\begin{document}

\title[Biharmonic functions on open manifolds]{The qualitative behavior for biharmonic functions on open manifolds}
\setlength{\baselineskip}{1.1\baselineskip}

\author[Wang]{Lin Wang}
\address{School of Mathematical Sciences, Shanghai Jiao Tong University\\ 800 Dongchuan Road \\ Shanghai, 200240 \\ P. R. China}%
\email{wanglin9833@sjtu.edu.cn}

\author[Zhu]{Miaomiao Zhu}
\address{School of Mathematical Sciences, Shanghai Jiao Tong University\\ 800 Dongchuan Road \\ Shanghai, 200240 \\ P. R. China}
\email{mizhu@sjtu.edu.cn}

\thanks{The second author would like to thank Professor William P. Minicozzi II for valuable advice and encouragement.}

\subjclass[2010]{31A30, 58J05}
\keywords{Biharmonic functions, nonnegative Ricci curvature, polynomial growth.}

\date{June 18, 2025}
\begin{abstract}
	For a complete noncompact Riemannian manifold with nonnegative Ricci curvature, we show that bounded biharmonic functions are constant and the space consists of biharmonic functions with polynomial growth of a fixed rate is finite dimensional.  Also, we derive a Weyl type bound for this space. 	Finally, we present a finite dimensional result for a class of fourth-order operators on $\mathbb{R}^n$  satisfying certain coefficient conditions.
\end{abstract}
\maketitle

\section{Introduction}\label{intro}
	Let $(M^n,g)$ be an $n$-dimensional Riemannian manifold with metric $g$ and it will be simply referred as $M^n$ or just $M$. Denote $\nabla$ and $\Delta$ to be the gradient operator and the Laplace-Beltrami operator on $M$ respectively. A function $u$ on $M$ is called harmonic if
	\[\Delta u=0,\]
	and it is called biharmonic if
	\begin{equation*}
		\Delta^2 u=0.
	\end{equation*}
	 
	Around 1974, Yau \cite{Yau1} established the famous Liouville type theorem for harmonic functions, namely, any positive harmonic function on complete noncompact Riemannian manifolds with nonnegative Ricci curvature must be constant. Later on, Cheng-Yau \cite{Cheng-Yau} generalized this result and established the so called Cheng-Yau gradient estimate. As a consequence of this gradient estimate, Cheng \cite{Cheng} showed that harmonic functions with sublinear growth on such manifolds must be constant. Based on these fundamental observations, Yau proposed the following conjecture (see e.g. \cite{Yau3,Yau4,Yau5} and \cite{Li1}): 
	
	\begin{conjecture1}
		For an open manifold with nonnegative Ricci curvature, the space of harmonic functions with polynomial growth of a fixed rate is finite dimensional.
	\end{conjecture1}
	Here, functions with polynomial growth with a fixed rate are defined as follows:	
	\begin{definition}\label{cond-biharmonic}
		Suppose $M^n$ is an open manifold, namely, a complete noncompact Riemannian manifold without boundary. Let $p \in M$ be a fixed point and $r(x)$ be the distance from $x$ to $p$. Then for $d\ge 0$, we say that a function $u$ on $M$ has polynomial growth of rate $d$ with respect to $p$ if 
		\begin{equation}\label{poly-grow}
			|u|(x) \le C\left(1+r^d(x)\right)
		\end{equation} 
		for some constant $0< C<\infty$.
	\end{definition}
	
	 Yau's conjecture has attracted lots of attention and there has been extensive and substantial work about this problem, see e.g. Kazdan\cite{Kazdan}, Li-Tam\cite{Li-Tam1,Li-Tam2}, Wu\cite{Wu}, Donnelly-Fefferman\cite{Donnelly-Fefferman}, Kassue\cite{Kasue1,Kasue2}, Li\cite{Li1,Li2}, Cheeger-Colding-Minicozzi\cite{Cheeger-Colding-Minicozzi}, Wang \cite{Wang}, Christiansen-Zworski\cite{Christiansen-Zworski}, Lin \cite{Lin} and Colding-Minicozzi\cite{Colding-Minicozzi2}. Yau's conjecture was completely solved in a celebrated work by Colding-Minicozzi\cite{Colding-Minicozzi1} and furthermore a sharp Weyl type upper bound for harmonic functions was derived in \cite{Colding-Minicozzi3}. Subsequently, Li \cite{Li3} provided another proof of Yau's conjecture. 

	As a natural generalization of harmonic functions, the theory of biharmonic functions can date back to 1860s when Airy \cite{Airy} and Maxwell \cite{Maxwell} investigated it to describe a mathematical model of elasticity. In fact, biharmonic functions have applications in a wide range of fields, including elasticity theory and fluid dynamics etc. Harmonic functions are always biharmonic, however, the reverse is in general not true. For instance, there exist manifolds of arbitrary dimension on which the space composed of proper biharmonic (biharmonic but not harmonic) functions is not empty, see e.g. classical works in the 1970s by Sario-Wang\cite{Sario-Wang2} and Sario-Nakai-Wang-Chung\cite{Sario-Nakai-Wang-Chung}. A specific and simple example of proper biharmonic function is the square distance function on $\mathbb{R}^n$, that is,
		\begin{equation*}
			u(x)=r^2(x)=\sum\limits_{i=1}^n x_i^2,
		\end{equation*}
	where $x=\left(x_1,\cdots, x_n\right) \in \mathbb{R}^n$. It is easy to check that $u$ is a biharmonic function but not harmonic, and $u$ satisfies the polynomial growth condition \eqref{poly-grow} with any $d \ge 2$.  
		
	On the other hand, there are many works aiming at looking for extra conditions under which biharmonicity implies harmonicity. For instance, Baird-Fardoun-Ouakkas\cite{Baird-Fardoun-Ouakkas} showed that biharmonic functions on complete manifolds with nonnegative Ricci curvature are harmonic if they have finite bienergy $\int_M |\Delta u|^2<\infty$. Later on,  Nakauchi-Urakawa-Gudmundsson\cite{Nakauchi-Urakawa-Gudmundsson} generalized this result by proving that biharmonic functions on a complete noncompact manifold $M$ with finite bienergy are harmonic in the case of ${\rm{Vol}}(M)=\infty$ (note that nonnegative Ricci curvature implies infinite volume, see e.g. \cite{Yau2} and \cite{Calabi}), and moreover, they showed that the finiteness of energy and bienergy of a biharmonic function $u$ on a complete manifold $M$ also implies that $u$ is harmonic. In this direction, Maeta\cite{Maeta} proved that a biharmonic function $u$ on a complete manifold $M$ with ${\rm{Vol}}(M)=\infty$ is harmonic if it has finite $q$-bienergy $\int_M |\Delta u|^q$ with some $2 \le q<\infty$ and biharmonicity of function $u$ on a complete manifold $M$ also implies harmonicity if $\int_M |\nabla u|^2<\infty$ and $\int_M |\Delta u|^q <\infty$ for some $2 \le q<\infty$. Furthermore, Luo \cite{Luo2} proved that a biharmonic function $u$ on a complete manifold $M$ is harmonic if ${\rm{Vol}}(M)=\infty$ and $\int_M |\Delta u|^q <\infty$ for some $1 < q < \infty$, or if $\int_M |\nabla u|^t < \infty$ and $\int_M |\Delta u|^q <\infty$ for some $1 \le t \le \infty$ and some $1 < q < \infty$.
	
It is well known that harmonic functions satisfy the Liouville type property  (see e.g. \cite{Edward, Moser, Yau1, Yau2}), namely, harmonic functions are constant under certain boundness conditions. In particular, bounded harmonic functions on open manifolds with nonnegative Ricci curvature are constant, see e.g. \cite{Cheng, Yau1}.  For biharmonic functions on open manifolds, however, relatively little is known.  

Around 1932, Nicolesco \cite{Nicolesco1} showed that every bounded biharmonic function on $\mathbb{R}^n$ is constant due to the Pizetti's formula, see also Nicolesco\cite{Nicolesco2} and Huilgol \cite{Huilgol}. In fact, there do exist complete noncompact manifolds of arbitrary dimension that carry nonconstant bounded biharmonic functions and there are also complete manifolds fail to carry nonconstant bounded biharmonic functions, see e.g. Sario-Wang\cite{Sario-Wang2}, Sario\cite{Sario1} and the references therein. Therefore, it is natural to impose the following:
	\begin{conjecture}
		Any bounded biharmonic function on open manifolds with nonnegative Ricci curvature is constant.
	\end{conjecture}
	
	Furthermore, the well known and classical fact shows that harmonic functions on the Euclidean space $\mathbb{R}^n$ with polynomial growth of a fixed rate are harmonic polynomials with degree at most of this fixed rate, then the space consists of such functions is finite dimensional and the exact dimension can be derived, see e.g.  Li \cite[Appendix B]{Li4}. Recall that the classical Almansi expression (see e.g. \cite[Proposition 1.3]{A-C-L}) shows that a biharmonic function on $\mathbb{R}^n$ can be written uniquely as 
	\[u(x)=h_1(x)+r^2(x)h_2(x),\]
	where $h_1(x)$ and $h_2(x)$ are harmonic functions and $r(x)$ is the Euclidean distance function on $\mathbb{R}^n$. Then one can conclude from the Almansi expression, the analysis in \cite[Appendix B]{Li4} and the result of Martinazzi \cite[Theorem 5]{Martinazzi} to obtain that biharmonic functions on $\mathbb{R}^n$ with polynomial growth of fixed rate $d \ge 0$ must be polynomials with degree at most $d$ and then the space consists of such functions is also of finite dimension with an upper bound of order $n-1$ with respect to the growth rate $d$ for $d \ge 1$. In 2007, Chen-Wang \cite{Chen-Wang} considered a class of $2m$-order ($m \ge 1$) linear partial differential equations on $\mathbb{R}^n$ and proved that the order of the dimension upper bound for the space of solutions with polynomial growth is $2mn$ if the equation is in nondivergence form and is $mn$ if in divergence form.
	
	In the past fifty years, however, relatively little is known about the framework of the space of biharmonic functions on general open manifolds. It is natural to impose the following extension of Yau's conjecture to the framework of biharmonic functions:
	
	\begin{conjecture2}
		For an open manifold with nonnegative Ricci curvature, the space of biharmonic functions with polynomial growth of a fixed rate is finite dimensional.
	\end{conjecture2}
	
	In the present paper, we shall firstly prove that any bounded biharmonic function on open manifolds with nonnegative Ricci curvature is constant. Then, we provide a resolution to the generalized Yau's conjecture and obtain a Weyl type bound of the space of biharmonic functions with polynomial growth on open manifolds.

	For a fixed point $p \in M$, define $\mathcal{H}^2_{p,d}(M)$ to be the linear space consists of biharmonic functions on $M$ with polynomial growth of rate $d \ge 0$ with respect to $p$, that is,
	\[\mathcal{H}^2_{p,d}(M)=\left\{u \in W^{2,2}_{loc}(M, \mathbb{R}) : \ \Delta^2u=0, \,|u|(x) \le C\left(1+r^d(x)\right)\right\},\] 
	where $r(x)$ is the distance function on $M$ from $p$. Denote by $B_r(p)$ the geodesic ball of radius $r$ centered at $p$, and we always omit the symbol of volume form in integrals for simplicity. Before stating our results, we recall some definitions on complete Riemannian manifolds:
	
	\begin{definition}\label{quadratic decay}
		A complete manifold $M^n$ is said to have asymptotically nonnegative Ricci curvature of quadratic decay at $p \in M$ if
		\begin{equation*}
			{\rm{Ric}}_x(\cdot,\cdot) \ge -(n-1) \frac{\kappa^2}{1+r^2(x)}\left<\cdot,\cdot\right>
		\end{equation*}
		holds for all $x \in M$ in the sense of quadratic forms with some constant $\kappa \ge 0$, where $r(x)$ is the distance from $x$ to $p$. 
	\end{definition}
	
	\begin{definition}\label{Laplacian}
		A complete manifold $M^n$ is said to have the Laplacian cut-off property at $p \in M$ if for all $r>0$ and $\alpha>1$, there exists a family of cut-off functions
		\[\varphi_r:M \to [0,1],\quad \varphi_r \in C_c^{\infty}(M)\]
		satisfying
		\[\varphi_r =1\ {\rm{on}}\ B_r(p),\quad {\rm{supp}}\,\varphi_r \subset B_{\alpha r}(p), \quad |\nabla \varphi_r| \le \frac{C_1}{r}\quad {\rm{and}}\quad |\Delta \varphi_r| \le \frac{C_2}{r^2}\]
		with constants $0<C_1,C_2<\infty$ independent of $r$. We say $M$ has the Laplacian cut-off property if it has the Laplacian cut-off property at every $x \in M$.
	\end{definition}
	
	\begin{remark}\label{RK1-1}
		Open manifolds with nonnegative Ricci curvature have the Laplacian cut-off property, see e.g. works by G\"uneysu \cite[Theorem 2.2]{Guneysu} and Bianchi-Setti \cite[Corollary 2.3]{Bianchi-Setti}. In fact, this is based on a highly nontrivial result from Riemannian rigidity theory, see e.g. Schoen-Yau \cite{Schoen-Yau1994}, Cheeger-Colding \cite{Cheeger-Colding}, Cheeger \cite{Cheeger}, and Wang-Zhu \cite{Wang-Zhu}. More generally, Bianchi-Setti \cite{Bianchi-Setti} proved that if $M$ has asymptotically nonnegative Ricci curvature of quadratic decay at $p \in M$, then it has the Laplacian cut-off property at $p$. For more explorations of open manifolds with asymptotically nonnegative Ricci curvature, see e.g. works by Zhu \cite{Zhu} and Pigola-Rigoli-Setti \cite{PRS}.
	\end{remark}
	
	\begin{definition}\label{doub-property}
		A complete manifold $M^n$ is said to have the doubling property if there exists a constant $0< C_D <\infty$ such that for all $p \in M$ and $r > 0$,
		\begin{equation}\label{doub-inequal}
			{\rm{Vol}}\left(B_{2r}(p)\right) \le C_D\,{\rm{Vol}}\left(B_r(p)\right),
		\end{equation}
		where ${\rm{Vol}}(B_r(p))$ is the volume of $B_r(p)$ with respect to the metric on the manifold.
	\end{definition}
	
	\begin{remark}\label{RK1-2}
		For any $p \in M$ and $r>0$, the doubling property implies a volume estimate for geodesic balls on $M$ as
		\[{\rm{Vol}}\left(B_r(p)\right) \le C_D {\rm{Vol}}\left(B_1(p)\right) \left(1+r^{\log_2 C_D}\right)\]
	\end{remark} 
	
	\begin{definition}\label{UNP}
		A complete manifold $M^n$ is said to satisfy the uniform Neumann-Poincar\'e inequality if there exists a constant $0< C_N < \infty$ such that for all $p \in M$, $r>0$ and $u \in W^{1,2}_{loc}(M)$,
		\begin{equation}\label{UNP-inequal}
			\int_{B_r(p)}(u-\overline{u})^2 \le C_N\,r^2\int_{B_r(p)}|\nabla u|^2,
		\end{equation}
		where $W^{1,2}_{loc}(M)$ is the local $(1,2)$-Sobolev space on $M$ and
		\begin{equation*}
			\overline{u}=\dfrac{\int_{B_r(p)} u}{{\rm{Vol}}\left(B_r(p)\right)}
		\end{equation*}
	\end{definition}
	
	\begin{remark}\label{RK1-3}
		 If $M^n$ is a complete manifold with nonnegative Ricci curvature, $M$ has the doubling property with constant $C_D = 2^n$ by the classical relative volume comparison theorem. Moreover, Buser\cite{Buser} showed that such manifold also satisfies the uniform Neumann-Poincar\'e inequality as an application of the inequality about isoperimetric constant.
	\end{remark}

	Our first result is a Liouville type theorem for biharmonic functions on open manifold, which is a natural generalization of the classical results of biharmonic function on $\mathbb{R}^n$ by Nicolesco \cite{Nicolesco1} and harmonic functions on open manifolds  \cite{Cheng}:
	\begin{theorem}\label{Biharmonic Liouville Theorem}
		Let $M^n$ be an open manifold with nonnegative Ricci curvature. Then any nonconstant biharmonic function on $M$ must grow at least linearly.
	\end{theorem}
	
	Next, we state the following finite dimensional result for biharmonic functions on open manifold:
	\begin{theorem}\label{mainThm}
		Let $M^n$ be an open manifold with asymptotically nonnegative Ricci curvature of quadratic decay at $p \in M$ with constant $\kappa \ge 0$, which satisfies the uniform Neumann-Poincar\'e inequality with constant $C_N$ and has the doubling property with constant $C_D$. Then for $d \ge 0$, the space $\mathcal{H}^2_{p,d}(M)$ is finite dimensional, which is bounded by a constant that depends on $n$, $\kappa$, $d$, $C_D$ and $C_N$.
	\end{theorem}
	
	Consequently, by Remark \ref{RK1-1} and Remark \ref{RK1-3}, we give an affirmative answer to the generalized Yau's conjecture.
	\begin{theorem}\label{specialThm}
		For any open manifold $M^n$ with nonnegative Ricci curvature, the space $\mathcal{H}^2_{p,d}(M)$ is finite dimensional for all $p \in M$ and $d \ge 0$.
	\end{theorem}
	
	Furthermore, we derive a Weyl type bound for the dimension of the space of biharmonic functions with polynomial growth on open manifolds with nonnegative Ricci curvature.
	\begin{theorem}\label{upperbound}
		Let $M^n$ be an open manifold with nonnegative Ricci curvature, then for any $p \in M$, we have
		\[\dim \mathcal{H}_{p,d}^2(M)=1\]
		for $0\le d<1$ and
		\[\dim \mathcal{H}_{p,d}^2(M) \le C d^{n-1}\]
		for $d \ge 1$, where $0<C = C(n) <\infty$ is a constant.
	\end{theorem}
	\begin{remark}
		As discussed before, in the case of Euclidean spaces, for any $p \in \mathbb{R}^n$ and $d \ge 1$, the upper bound for $\dim\mathcal{H}^2_{p,d}(\mathbb{R}^n)$ is $C d^{n-1}$. This indicates that the rate of growth given in Theorem \ref{upperbound} is  sharp compared to the Euclidean case.
	\end{remark}
	
	Finally, we present a finite dimensional result for a class of fourth-order partial differential operators on $\mathbb{R}^n$  satisfying certain coefficient conditions.
	
	\begin{theorem}\label{EDF1}
		Let $\left\{a_{ij}(x)\right\}_{i,j=1}^n$ be measurable functions on $\mathbb{R}^n$  and consider the operator $\mathcal{L}$ defined by
		\begin{equation}\label{EQ4-1}
			\mathcal{L}u:=\sum\limits_{i,j,k,l=1}^n \frac{\partial}{\partial x_k}\left(a_{kl}(x)\frac{\partial^2}{\partial x_l\partial x_i}\left(a_{ij}(x)\frac{\partial}{\partial x_j} u\right)\right).
		\end{equation}
		Assume that $a_{ij}(x) \in C^1\left(\mathbb{R}^n\backslash \overline{B_R}(0)\right)$ for some $R>0$ and there are constants $c_i>0$ ($i=1,\cdots,4$) such that
		\begin{equation}\label{EQ4-2}
			a_{ij}(x)=a_{ji}(x),\quad c_1\,|\xi|^2 \le \sum\limits_{i,j=1}^n a_{ij}(x)\,\xi_i\,\xi_j  \le c_2\,|\xi|^2
		\end{equation}
		for any $\xi=(\xi_1,\cdots,\xi_n) \in \mathbb{R}^n$ and
		\begin{equation}\label{EQ4-3}
			\left|a_{ij}(x)\right| \le c_3, \quad \left|\sum\limits_{i,j=1}^n\partial_i\, a_{ij}(x)\right| \le \frac{c_4}{|x|}
		\end{equation}
		for $|x| > R$. Then for all $d\ge 0$, there is a constant $0<C<\infty$ depending only on $n$, $d$ and $\left\{c_i\right\}_{i=1}^4$ such that
		\begin{equation*}
			\dim\,\mathcal{H}^2_{d}(\mathbb{R}^n,\mathcal{L}) \le C,
		\end{equation*}
		where $\mathcal{H}^2_{d}(\mathbb{R}^n,\mathcal{L})$ is the space of weak solutions to $\mathcal{L} u=0$ with polynomial growth of rate $d$ with respect to the original point $0 \in \mathbb{R}^n$. 
	\end{theorem}
	\begin{remark} 
\
		\begin{enumerate}[(1)]
			\item The proof in \cite{Colding-Minicozzi3} also implies that the space $\mathcal{H}_d(\mathbb{R}^n, L)$ in Theorem \ref{CM3} has Weyl type upper bound $Cd^{n-1}$. Then applying similar analysis as in Section \ref{Dimensionbound}, we have that the upper bound of $\dim \mathcal{H}_d^2(\mathbb{R}^n, \mathcal{L})$ is also $Cd^{n-1}$. 
			\item Note that the fourth-order operator $\mathcal{L}$ considered in Theorem \ref{EDF1} is in general different from the fourth-order operators  in \cite{Chen-Wang}. These two operators coincide when the coefficients $\{a_{ij}\}$ are constant, in this special case, the estimate for $\dim \mathcal{H}_d^2(\mathbb{R}^n, \mathcal{L})$ in Theorem \ref{EDF1} is $Cd^{n-1}$ while the estimate derived in \cite{Chen-Wang} is $Cd^{2n}$.
		\end{enumerate}
	\end{remark}

	The rest of this paper is organized as follows: In Section \ref{RevPoin}, we prove a bienergy estimate for biharmonic functions and show Theorem \ref{Biharmonic Liouville Theorem}. In Section \ref{Dimensionbound}, we shall prove the finiteness of dimension for the space as stated in Theorem \ref{mainThm} and then derive the sharp upper bound showed in Theorem \ref{upperbound}. Finally, we show Theorem \ref{EDF1} in Section \ref{Apply}.

\vskip1cm
\section{key estimate and Liouville type theorem of biharmonic functions}\label{RevPoin}
	Around 1974, Yau derived a reverse Poincar\'e inequality for harmonic functions on open manifolds in \cite{Yau2}, which is important in the proof of his conjecture on harmonic functions eventually settled in \cite{Colding-Minicozzi1}. In this section, we shall establish a higher order version of this inequality, namely, a key estimate for biharmonic functions on open manifolds, which plays important roles in deriving the Liouville type result for biharmonic functions - Theorem \ref{Biharmonic Liouville Theorem} and the finite dimensional result for biharmonic functions - Theorem \ref{mainThm}.

\begin{lemma}\label{RP-inqual}
	Let $M^n$ be an open manifold having asymptotically nonnegative Ricci curvature of quadratic decay at $p$ with constant $\kappa \ge 0$ and $u$ be a biharmonic function on $M$. Then there is a constant $R_0>0$ such that for all $r>R_0$ and $\alpha>1$,
	\begin{equation}\label{EQ2-1}
		r^2\int_{B_r(p)}|\Delta u|^2 \le C\int_{B_{\alpha r}(p)}u^2,
	\end{equation}
	where $C>0$ is a constant depending only on $n$, $\kappa$ and $\alpha$.
\end{lemma}

\begin{proof}[Proof]
	By Definition \ref{quadratic decay} and \cite[Corollary 2.3]{Bianchi-Setti}, the manifold $M$ has the Laplacian cut-off property at $p$. Hence, for $r>0$ and $\alpha >1$, there exists a smooth function $\varphi$ on $M$ satisfying:
	\begin{equation*}
		\begin{cases}
			0 \le \varphi(x) \le 1 \\
			\varphi(x)=1, \qquad \quad  {\rm{for}} \    x \in B_r(p) \\
			\varphi(x)=0, \qquad \quad   {\rm{for}} \     x \in M \backslash B_{\alpha r}(p) \\
			|\nabla \varphi|(x) \le \frac{C_1}{r} \\
			|\Delta \varphi|(x) \le \frac{C_2}{r^2}
		\end{cases}
	\end{equation*}
with constants $0 < C_i=C_i(n,\kappa,\alpha) <\infty$ ($i=1,2$) being independent of $r$. Then, there holds
	\begin{equation*}\label{3-4}
		\begin{aligned}
			0= \int_M \varphi^4\,u\,\Delta^2u 
			=\int_M \left(\varphi^4\,\Delta u+4\,\varphi^3\,u\,\Delta \varphi+8\,\varphi^3\,\nabla\,\varphi\,\nabla\,u +12\,\varphi^2\,u\,|\nabla \varphi|^2\right)\, \Delta u,
		\end{aligned}
	\end{equation*}
	that is,
	\[\int_M \varphi^4\,|\Delta u|^2 = -8\int_M \varphi^3\,\Delta u\,\nabla\varphi\,\nabla u-12\int_M\varphi^2\,u\,\Delta u\,|\nabla\varphi|^2-4\int_M\varphi^3\,u\,\Delta u\,\Delta \varphi.\]
	By H\"older inequality,
	\begin{align*}
		&\int_M \varphi^4\,|\Delta u|^2 \\
		\le&\ 8\int_M \varphi^3 \left|\nabla\varphi\right| \left|\nabla u\right| \left|\Delta u\right|+12\int_M\varphi^2 \left|u\right| \left|\nabla\varphi\right|^2 \left|\Delta u\right|+4\int_M\varphi^3 \left|u\right| \left|\Delta\varphi\right| \left|\Delta u\right|\\
		\le&\ \left\{ 8\left(\int_M\varphi^2 \left|\nabla\varphi\right|^2 \left|\nabla u\right|^2\right)^{\frac{1}{2}}+12\left(\int_M \left|\nabla\varphi\right|^4\,u^2\right)^{\frac{1}{2}} +4\left(\int_M \varphi^2 \left|\Delta\varphi\right|^2\,u^2\right)^{\frac{1}{2}} \right\} \times \\
		&\ \left( \int_M \varphi^4 \left|\Delta u\right|^2 \right)^{\frac{1}{2}},
	\end{align*}
	then it follows from the mean inequality and the properties of  the cut-off function $\varphi$ that
	\begin{equation}\label{EQ2-2}
		\begin{aligned}
			\int_M \varphi^4\,|\Delta u|^2
			&\le\ 480\left\{\int_M \varphi^2 \left|\nabla\varphi\right|^2\left|\nabla u\right|^2+ \int_M\left(\left|\nabla\varphi\right|^4+\left|\Delta\varphi\right|^2 \right)u^2 \right \}\\
			&\le\ \frac{480\left(C_1^4+C_2^2\right)}{r^4}\int_{B_{\alpha r}(p)}u^2+\frac{480C_1^2}{r^2}\int_{M}\varphi^2 \left|\nabla u\right|^2.
		\end{aligned}
	\end{equation}
	Furthermore, 
	\begin{align*}
		\int_M \varphi^2\,|\nabla u|^2
		&=\ -\int_M \varphi^2\,u\,\Delta u-2\int_M\varphi\,u\,\nabla\varphi\,\nabla u \\
		&\le\ \frac{1}{2}\int_M \varphi^2 \left|\nabla u\right|^2+2\int_M \left|\nabla\varphi\right|^2 u^2+\frac{1}{2}\int_M \varphi^4\left|\Delta u\right|^2+\frac{1}{2}\int_{B_{\alpha r}(p)} u^2,
	\end{align*}
	that is,
	\begin{equation}\label{EQ2-3}
		\int_M \varphi^2 \left|\nabla u\right|^2 \le \int_M \varphi^4\left|\Delta u\right|^2+4\int_M \left|\nabla \varphi\right|^2\,u^2+\int_{B_{\alpha r}(p)} u^2.
	\end{equation}
	Take $R_0=\max\left\{8\sqrt{15}C_1,1\right\}>0$, then for $r>R_0$, \eqref{EQ2-2} and \eqref{EQ2-3} imply that
	\begin{align*}
		\int_M \varphi^4 \left|\Delta u\right|^2 \le &\ \frac{1}{2}\int_M \varphi^4 \left|\Delta u\right|^2+ \left(\frac{480\left(C_1^4+C_2^2\right)}{r^4}+\frac{480C_1^2}{r^2}\right)\int_{B_{\alpha r}(p)}u^2 \\
		&\ +\frac{1920C_1^2}{r^2}\int_{M}\left|\nabla\varphi\right|^2\,u^2.
	\end{align*}
	Therefore, use the definition of $\varphi$ again, since $r >R_0 \ge 1$, it follows that there exists a constant $C>0$ depending only on $n$, $\kappa$ and $\alpha$ such that
	\[\int_{B_r(p)}|\Delta u|^2 \le \int_{M}\varphi^4\,|\Delta u|^2\le \frac{C}{r^2}\int_{B_{\alpha r}(p)}u^2 \]
	and the proof is completed.
\end{proof}

Letting $r \to \infty$, the key estimate in Lemma \ref{RP-inqual} indicates that any global $L^2$-integrable biharmonic function on open manifolds with asymptotically nonnegative Ricci curvature of quadratic decay must be harmonic, and hence it has to be constant since there is no nonconstant global $L^p$ harmonic function on open manifolds for any $p \in (1,\infty)$ \cite[Theorem 3]{Yau2}. In particular, it has to be identically $0$ if the Ricci curvature is nonnegative \cite[Theorem 7]{Yau2}. We conclude this as follows:

\begin{corollary}
		Let $u$ be a biharmonic function defined on an open manifold $M$ with asymptotically nonnegative Ricci curvature of quadratic decay. Then either 
		\[\int_M u^2 =\infty\]
		or $u$ is a constant. In particular, any global $L^2(M)$-integrable biharmonic function must be identically $0$ if the Ricci curvature of $M$ is nonnegative.
	\end{corollary}

	To get the Liouville type property for biharmonic functions on open manifolds stated in Theorem \ref{Biharmonic Liouville Theorem}, we need the following mean value inequality by Li-Schoen \cite{Li-Schoen} and the stronger Liouville type theorem for harmonic functions by Cheng \cite{Cheng} (see also e.g. \cite[Corollary 1.5]{Li5}):
	\begin{theorem}[{\cite[Theorem 1.2]{Li-Schoen}}]\label{LS}
		Let $M^n$ be a complete Riemannian manifold without boundary and assume that the Ricci curvature of $M$ is bounded from below by $-(n-1)K$ for some constant $K \ge 0$. Let $f$ be a nonnegative subharmonic function on $M$. Then there is a constant $C>0$ depending only on $n$ such that for any $r>0$, $p\in M$ and $\tau \in (0,1/2)$, we have 
		\[\sup\limits_{B_{(1-\tau)r}(p)} f^2 \le \frac{\tau^{-C\left(1+\sqrt{K}r\right)}}{{\rm{Vol}}\left(B_r(p)\right)}\int_{B_r(p)} f^2.\]
	\end{theorem}
	\begin{theorem}[\cite{Cheng}]\label{Cheng}
		Let $M^n$ be an open manifold with nonnegative Ricci curvature, then $M$ does not admit any nonconstant harmonic function $u$ satisfying
		\[r^{-1}(x)\,\left|u\right|(x) \to 0\]
		as $x \to \infty$, where $r(x)$ is the distance function on $M$ from a fixed point.
	\end{theorem}
	
	\begin{proof}[\textbf{Proof of Theorem \ref{Biharmonic Liouville Theorem}:}]
		Let $r(x)$ be the distance function from $p \in M$ and $u$ be a biharmonic function on $M$ with 
		\[|u|(x) \le B\left(1+r^s(x)\right)\] 
		for some $0<B<\infty$ and $s < 1$, then by Lemma \ref{RP-inqual},  $u$ satisfies the key estimate \eqref{EQ2-1} for any $r> R_0 \ge 1$ and $\alpha>1$. Since $\Delta u$ is harmonic, it follows from Theorem \ref{LS} that there is a constant $C>0$ depending only on $n$ such that
		\[\sup\limits_{B_r(p)}|\Delta u|^2 \le \frac{C}{{\rm{Vol}}(B_{2r}(p))} \int_{B_{2r}(p)}|\Delta u|^2 \le \frac{B^2\,C\,{\rm{Vol}}(B_{4r}(p))\left(1+r^{2s}\right)}{r^2 {\rm{Vol}}(B_{2r}(p))} \le \frac{2^n\,B^2\, C\,r^{2s}}{r^2},\]
		where the last inequality follows from the Bishop volume comparison theorem. Therefore, by letting $r \to \infty$, we have
		\[\Delta u =0\] 
		on $M$. Theorem \ref{Cheng} then shows that $u$ is constant.
	\end{proof}

\vskip 1cm
\section{Weyl type bound for biharmonic functions}\label{Dimensionbound}

In this section, we shall firstly prove the finite dimensional result for biharmonic functions on open manifolds - Theorem \ref{mainThm} and then derive a Weyl type sharp upper bound for the dimension of the space of biharmonic functions on open manifolds - Theorem \ref{upperbound}.
	
To prove our main results, we firstly recall some classical results about the space of harmonic functions. 	Let $\mathcal{H}_{p,d}(M)$ be the space of harmonic functions on $M$ with polynomial growth of rate $d$ with respect to $p \in M$. 
	
	\begin{theorem}[{\cite[Theorem 0.7]{Colding-Minicozzi1}}]\label{CM1}
		If $M^n$ is an open manifold which has the doubling property with constant $C_D$ and satisfies the uniform Neumann-Poincar\'e inequality with constant $C_N$, then for all $p \in M$ and $d \ge 0$, there exists a constant $0<C< \infty$ depending only on $n$, $d$, $C_D$, and $C_N$ such that 
		\[\dim\mathcal{H}_{p,d}(M) \le C.\]
	\end{theorem}

	\begin{theorem}[{\cite[Corollary 0.10]{Colding-Minicozzi3}}]\label{CM2}
		If $M^n$ is an open manifold with nonnegative Ricci curvature, then there exists a constant $0<C<\infty$ depending only on $n$ such that for all $p \in M$ and $d \ge 1$,
		\[\dim\mathcal{H}_{p,d}(M) \le C\,d^{n-1}.\]
	\end{theorem}
	
Next, we shall associate these results with the space $\mathcal{H}^2_{p,d}(M)$ defined with respect to a chosen point $p \in M$. Since it is easy to see that $\mathcal{H}_{p,d}(M) \subset \mathcal{H}^2_{p,d}(M)$, we consider the following linear space:
	\[\widetilde{\mathcal{H}}_{p,d}^2(M):=\left\{\Delta u: u \in \mathcal{H}_{p,d}^2(M)\right\},\]
	and then define a map $\Phi$ as follows:
	\begin{align*}
		\Phi: \mathcal{H}_{p,d}^2(M) & \longrightarrow  \widetilde{\mathcal{H}}_{p,d}^2(M) \\
		u & \longmapsto  \Delta u .
	\end{align*}
	It is clear that functions in $\widetilde{\mathcal{H}}_{p,d}^2(M)$ are harmonic and $\Phi$ is a linear surjection from $\mathcal{H}_{p,d}^2(M)$ to $\widetilde{\mathcal{H}}_{p,d}^2(M)$ with \[\ker(\Phi)=\mathcal{H}_{p,d}(M)\quad {\rm{and}}\quad {\rm{Im}}(\Phi)=\widetilde{\mathcal{H}}_{p,d}^2(M).\]
	We will denote by $C>0$ a constant which may be different from line to line.
	
	\begin{proof}[\textbf{Proof of Theorem \ref{mainThm}}]
		
		Firstly, Theorem \ref{CM1} tells us that $\mathcal{H}_{p,d}(M)$ is of finite dimension with a dimension upper bound $C(n, d, C_D, C_N)>0$. 

Furthermore, by Lemma \ref{RP-inqual} and Remark \ref{RK1-2}, for any $u \in \mathcal{H}_{p,d}^2(M)$, there exist constants $0<C=C(n,\kappa,\alpha)<\infty$ and $R_0 \ge 1$ such that
		\begin{equation*}\label{EQ3-1}
		\begin{aligned}
			\int_{B_r(p)}|\Delta u|^2 
			\le  C\left(1+r^{2d+\log_2C_D-2}\right)
		\end{aligned}
		\end{equation*}
		for $r>R_0$ and $\alpha>1$. Take a specific $\alpha>1$, then for any $f \in \widetilde{\mathcal{H}}_{p,d}^2(M)$ and $r>R_0$, there exists a constant $0<C<\infty$ depending only on $n$ and $\kappa$ such that
		\[\int_{B_r(p)} f^2 \le C\left(1+r^{2d+\log_2C_D-2}\right).\]
		Since $f \in \widetilde{\mathcal{H}}_{p,d}^2(M)$ is harmonic, the proof of Theorem \ref{CM1} then implies that $\widetilde{\mathcal{H}}_{p,d}^2(M)$ is also a finite dimensional space with a dimension upper bound $0<C<\infty$ depending only on $n$, $\kappa$, $d$, $C_D$ and $C_N$.
		
		Now we can claim that $\mathcal{H}^2_{p,d}(M)$ is of finite dimension. To see this, assume 
		\[\dim \mathcal{H}^2_{p,d}(M)=\infty\]
		and let $\{v_1,\cdots,v_l\}$ be the basis of $\mathcal{H}_{p,d}(M)$. Then for any positive integer $s$, we can find $s$ linearly independent elements $\{w_1,\cdots,w_s\}$ in $\mathcal{H}^2_{p,d}(M)$ with $v_1,\cdots,v_l, w_1, \cdots, w_s$ being also linearly independent. Let 
		\[z_i=\Delta w_i, \quad i=1,\cdots,s,\] 
		then $\{z_1, \cdots,z_s\} \subset \widetilde{\mathcal{H}}_{p,d}^2(M)$. If $z_1, \cdots,z_s$ are linearly dependent, then there exists 
		\[0 \ne (k_1,\cdots,k_s) \in \mathbb{R}^s\] 
		such that
		\[k_1 z_1+\cdots+k_s z_s=0,\]
		which implies that $k_1 w_1+\cdots+k_s w_s$ is an element in $\mathcal{H}_{p,d}(M)$ and thus there exists 
		\[(h_1,\cdots, h_l) \in \mathbb{R}^l\] 
		such that
		\begin{equation}\label{EQ3-3}
			k_1 w_1+\cdots+k_s w_s+h_1 v_1+\cdots+h_l v_l=0.
		\end{equation}
		Since $(k_1,\cdots,k_s) \ne 0$ and $v_1,\cdots,v_l, w_1, \cdots, w_s$ are linearly independent, the identity \eqref{EQ3-3} is impossible. Hence $z_1, \cdots,z_s$ are linearly independent in $\widetilde{\mathcal{H}}_{p,d}^2(M)$. However, since $\widetilde{\mathcal{H}}_{p,d}^2(M)$ is of finite dimension and $s$ is an arbitrary positive integer, this is a contradiction. 

Therefore, the space $\mathcal{H}^2_{p,d}(M)$ is finite dimensional. The dimensional formula in linear algebraic theory then shows that
		\[\dim\mathcal{H}_{p, d}^2(M)=\dim \ker(\Phi)+\dim {\rm{Im}}(\Phi)=\dim\mathcal{H}_{p, d}(M)+\dim\widetilde{\mathcal{H}}_{p, d}^2(M)\]
		and the proof is completed.	
	\end{proof}
	
\begin{proof}[\textbf{Proof of Theorem \ref{upperbound}}]
By Theorem \ref{Biharmonic Liouville Theorem}, we have $\dim \mathcal{H}^2_{p,d}(M)=1$ for $0 \le d <1$. Hence, we only need to consider the case of $d \ge 1$.
		
		Since $M$ has nonnegative Ricci curvature, then Lemma \ref{RP-inqual} holds for all $p \in M$. From the proof of Theorem \ref{mainThm}, for any $p \in M$, since
		\begin{equation}\label{EQ3-4}
			\dim\mathcal{H}^2_{p,d}(M) =\dim\mathcal{H}_{p,d}(M)+\dim\widetilde{\mathcal{H}}^2_{p,d}(M),
		\end{equation}
		the dimension estimate for $\mathcal{H}^2_{p,d}(M)$ is therefore reduced to estimating $ \dim\mathcal{H}_{p,d}(M)$ and $\dim\widetilde{\mathcal{H}}_{p,d}^2(M)$.
		
		To begin with, Theorem \ref{CM2} shows that there is a constant $C>0$ depending only on $n$ such that
		\[\dim\mathcal{H}_{p,d}(M) \le C\,d^{n-1}\]
		for $d \ge 1$.Then it remains to estimate ${\rm{dim}}\widetilde{\mathcal{H}}^2_{p,d}(M)$.
		
		For any $f \in \widetilde{\mathcal{H}}_{p,d}^2(M)$, as shown in the proof of Theorem \ref{mainThm}, take a specific $\alpha>1$, it follows that there exist constants $C=C(n)>0$ and $R_0 \ge 1$ such that
		\[\int_{B_r(p)} f^2\le C\left(1+r^{2d+n-2}\right)\]
		for $r >R_0$. On the other hand, since $f \in \widetilde{\mathcal{H}}_{p,d}^2(M)$ is harmonic, it follows that $\left|f\right|$ is a nonnegative subharmonic function on $M$, the mean value inequality stated in Theorem \ref{LS} then implies that $f$ has sublinear growth if $1 \le d <2$. Hence, Theorem \ref{Cheng} implies that
		\[\dim\widetilde{\mathcal{H}}_{p,d}^2(M)=1\] 
		for $1 \le d <2$. If $d \ge 2$, the proof of Theorem \ref{CM2} implies that there is a constant $C>0$ depending only on $n$ such that
		\[\dim\widetilde{\mathcal{H}}_{p,d}^2(M) \le C\left(d-1\right)^{n-1}.\]
		
		Above all, Theorem \ref{upperbound} follows from \eqref{EQ3-4}.
	\end{proof}

\vskip 1cm
\section{Finite dimension result for a class of fourth-order operator on Euclidean space}\label{Apply}
	In this section, we shall prove a finite dimensional result for a class of fourth-order partial differential operators on $\mathbb{R}^n$ - Theorem \ref{EDF1}.

	Firstly, we recall the following classical result for a class of second-order partial differential operators on $\mathbb{R}^n$:
	\begin{theorem}[{\cite[Theorem 6.8]{Colding-Minicozzi1}}]\label{CM3}
		Let $L$ be a second-order operator on $\mathbb{R}^n$ defined by
		\[Lu:=\sum\limits_{i,j=1}^n \frac{\partial}{\partial x_i}\left(a_{ij}(x) \frac{\partial}{\partial x_j} u\right)\]
		with some measurable functions $a_{ij}(x)$ ($i,j=1,\cdots,n$) satisfying
		\begin{equation*}
			a_{ij}(x)=a_{ji}(x),\quad \left(a_{ij}(x)\right) \ge c_1 I \quad {\rm{and}}\quad \sum\limits_{i,j=1}^n a_{ij}(x)\, x_i\,x_j \le c_2\,|x|^2
		\end{equation*} 
		for positive constants $c_1$ and $c_2$. Let $\mathcal{H}_d(\mathbb{R}^n, L)$ be the space composed of functions $u$ satisfying $L u=0$ on $\mathbb{R}^n$ in the weak sense and having polynomial growth with  rate $d$ with respect to the original point $0 \in \mathbb{R}^n$. Then for all $d \ge 0$, there is a constant $0<C<\infty$ depending only on $n$, $c_1$, $c_2$ and $d$ such that
		\[\dim \mathcal{H}_d(\mathbb{R}^n, L) \le C.\]
	\end{theorem}
	
	Let $\nabla u=\left(\frac{\partial}{\partial_1} u,\cdots,\frac{\partial}{\partial_n} u\right)^{\rm{T}}$ be the gradient of $u$ on $\mathbb{R}^n$. For measurable functions $a_{ij}(x)$ ($i,j=1,\cdots,n$) and (weakly) differentiable function $u$ on $\mathbb{R}^n$, for simplicity, we denote 
	\[A(x)=\left(a_{ij}(x)\right)_{n \times n},\quad\nabla^A u:=A\,\nabla u,\quad\Delta^A u:={\rm{div}}\,\nabla^A u,\] 
	where ${\rm{div}}$ is the divergence operator on $\mathbb{R}^n$. Then
	\[\Delta^A u=\sum\limits_{i,j=1}^n \partial_i\left(a_{ij}(x)\,\partial_j u\right) \quad {\rm{and}}\quad \mathcal{L}u=\Delta^A\Delta^A u.\]

	\begin{proof}[\textbf{Proof of Theorem \ref{EDF1}}]
		Let $B_r:=B_r(0)$ be the geodesic ball on $\mathbb{R}^n$ centered at $0$ and with radius $r$. From Theorem \ref{CM3} and the proof of Theorem \ref{mainThm}, it is sufficient to prove that there exists constants $R_0 \ge 0$ and $0<C<\infty$ such that any solution to equation \eqref{EQ4-1} satisfies
		\begin{equation}\label{RPI}
			\int_{B_r} \left| \Delta^A u \right|^2 \,dx \le \frac{C}{r^2}\int_{B_{\alpha r}} u^2\,dx.
		\end{equation}
		for all $r >R_0$ and $\alpha >1$.
	
		Firstly, for $r>R$ and $\alpha>1$, we can choose a smooth function $\varphi$ on $\mathbb{R}^n$ satisfying
		\begin{align*}
			\begin{cases}
				0 \le \varphi(x) \le 1 \\
				\varphi(x)=1 \quad&\ {\rm{for}} \  x \in B_r \\
				\varphi(x)=0 \quad&\ {\rm{for}} \  x \in \mathbb{R}^n \backslash B_{\alpha r} \\
				\left|\nabla \varphi\right| \le \frac{C_1}{r} \\
				\left|\nabla^2 \varphi\right| \le \frac{C_2}{r^2}
			\end{cases}
		\end{align*}
		with $0<C_i=C_i(n,\alpha)<\infty$ ($i=1,2$) independent of $r$. 
		From the definition of $\varphi$ and condition \eqref{EQ4-3}, there is a constant $D>0$ depending only on $n$, $\alpha$, $c_3$ and $c_4$ such that
		\[\left|\Delta^A \varphi\right| \le \frac{D}{r^2}.\]
		Then for any solution $u$ to \eqref{EQ4-1}, from the principle of integral by parts, we have that
		\begin{align*}
			0
			=&\ \int_{\mathbb{R}^n}\varphi^4\,u\,\mathcal{L} u \,dx \\
			=&\ \int_{\mathbb{R}^n}\sum\limits_{i,j,k,l=1}^n \varphi^4\,u\,\partial_k\left( a_{kl}\,\partial_l\,\partial_i\left(a_{ij}\,\partial_j u\right)\right) \,dx\\
			=&\ \int_{\mathbb{R}^n}\sum\limits_{i,j,k,l=1}^n  \partial_k\left( a_{kl}\,\partial_l\,(\varphi^4\,u)\right)\partial_i\left(a_{ij}\,\partial_j u\right) \,dx \\
			=&\ \int_{\mathbb{R}^n} \Delta^{A}\left(\varphi^4 u\right)\Delta^A u \,dx \\
			=&\ \int_{\mathbb{R}^n} \left(\varphi^4\,\Delta^A u+4\,\varphi^3\,u\,\Delta^A \varphi+8\,\varphi^3\,\nabla^A\varphi\,\nabla u+12\,\varphi^2\,u\,\nabla^A\varphi\,\nabla \varphi\right) \Delta^A u \,dx.
		\end{align*}
		From \eqref{EQ4-2} and the Cauchy-Schwarz inequality, since
		\begin{align*}
			\left|\nabla^A\varphi\,\nabla v\right| =  \left|\nabla\varphi\,\nabla^A v\right| =  \left|\sum\limits_{i,j=1}^n a_{ij}\,\partial_i\varphi\,\partial_j v \right| &\ \le  \  \left(\sum\limits_{i,j=1}^n a_{ij}\,\partial_i\varphi\,\partial_j \varphi \right)^{\frac{1}{2}}\left(\sum\limits_{i,j=1}^n a_{ij}\,\partial_i v\,\partial_j v \right)^{\frac{1}{2}} \\
			&\ \le   \  c_2\left|\nabla\varphi\right|\left|\nabla v\right|
		\end{align*}
		holds for any differentiable function $v$, we have that
		\begin{equation*}
			\int_{\mathbb{R}^n} \varphi^2\,\left|u\right|\,\left|\nabla^A\varphi\,\nabla\varphi\right|\left|\Delta^A u\right| \,dx
			\le c_2\int_{\mathbb{R}^n} \varphi^2\,\left|u\right|\,\left|\nabla \varphi\right|^2\,\left|\Delta^A u\right| \,dx.
		\end{equation*}
		and
		\begin{equation*}
			\int_{\mathbb{R}^n} \varphi^3\,\left|\nabla^A \varphi \,\nabla u\right|\,\left|\Delta^A u\right| \,dx \le c_2\int_{\mathbb{R}^n} \varphi^3\,\left|\nabla\varphi\right|\,\left|\nabla u\right|\,\left|\Delta^A u\right| \,dx.
		\end{equation*}
		Therefore,
		\begin{align*}
			&\int_{\mathbb{R}^n} \varphi^4\,\left|\Delta^A u\right|^2 \,dx \\
			=&\ -\int_{\mathbb{R}^n} \left(4\,\varphi^3\,u\,\Delta^A \varphi+8\,\varphi^3\,\nabla^A\varphi\,\nabla u+12\,\varphi^2\,u\,\nabla^A \varphi\,\nabla\varphi\right)\Delta^A u \,dx \\
			\le&\ \int_{\mathbb{R}^n} \left(4\,\varphi^3\,\left|u\right|\left|\Delta^A \varphi\right|\left|\Delta^A u\right|+8\,c_2\,\varphi^3 \left|\nabla \varphi\right|\left|\nabla u\right|\left|\Delta^A u\right|+12\,c_2\,\varphi^2\left|u\right|\left|\nabla \varphi\right|^2 \left|\Delta^A u\right|\right)\,dx\\
			\le&\ \left\{4\left(\int_{\mathbb{R}^n}\varphi^2\left|\Delta^A\varphi\right|^2 u^2\,dx\right)^{\frac{1}{2}}+8\,c_2\left(\int_{\mathbb{R}^n}\varphi^2\left|\nabla \varphi\right|^2\left|\nabla u\right|^2\,dx\right)^{\frac{1}{2}} \right.\\
			&\ \left. +12\,c_2\,\left(\int_{\mathbb{R}^n}\left|\nabla \varphi\right|^4 u^2\,dx\right)^{\frac{1}{2}}\right\}\, \times \left(\int_{\mathbb{R}^n} \varphi^4\,\left|\Delta^A u\right|^2\,dx\right)^{\frac{1}{2}},
		\end{align*}
		which implies that there exist a constant $D_1>0$ depending only on $c_2$ such that
		\begin{equation}\label{EQ4-5}
			\begin{aligned}
				&\int_{\mathbb{R}^n} \varphi^4\left| \Delta^A u \right|^2 \,dx \\
				\le&\ D_1 \left(\int_{\mathbb{R}^n}\varphi^2\left|\Delta^A\varphi\right|^2 u^2\,dx+\int_{\mathbb{R}^n}\left|\nabla\varphi\right|^4 u^2\,dx+\int_{\mathbb{R}^n}\varphi^2\left|\nabla\varphi\right|^2\left|\nabla u\right|^2\,dx\right) \\
				\le&\ D_2\left(\frac{1}{r^4}\int_{B_{\alpha r}} u^2\,dx+\frac{1}{r^2}\int_{\mathbb{R}^n} \varphi^2\,\left|\nabla u\right|^2\,dx\right),
			\end{aligned}
		\end{equation}
		where $D_2>0$ is a constant depending only on $n$, $\alpha$, $c_2$, $c_3$, and $c_4$. Furthermore, since
		\begin{align*}
			&\ \int_{\mathbb{R}^n} \varphi^2\left|\nabla u\right|^2\,dx \\
			\le&\ \frac{1}{c_1}\int_{\mathbb{R}^n} \varphi^2\, \nabla u\,\nabla^A u \,dx\\
			=&\ -\frac{1}{c_1}\int_{\mathbb{R}^n} \varphi^2\, u\,\Delta^A u \,dx-\frac{2}{c_1} \int_{\mathbb{R}^n} \varphi\, u\,\nabla\varphi\,\nabla^A u \,dx\\
			\le&\ \frac{1}{c_1}\int_{\mathbb{R}^n} \varphi^2\, |u|\,\left|\Delta^A u\right| \,dx+\frac{2c_2}{c_1} \int_{\mathbb{R}^n} \varphi\, |u|\,\left|\nabla\varphi\right|\left|\nabla u \right| \,dx\\ 
			\le&\ \frac{1}{2} \int_{\mathbb{R}^n} \varphi^2\left|\nabla u\right|^2\,dx+\frac{1}{2} \int_{\mathbb{R}^n} \varphi^4\left|\Delta^A u\right|^2\,dx+\frac{2c_2^2}{c_1^2} \int_{\mathbb{R}^n} \left|\nabla \varphi\right|^2u^2\,dx+\frac{1}{2c_1^2} \int_{B_{\alpha r}} u^2\,dx,
		\end{align*}
		where the first inequality follows from \eqref{EQ4-2}, this implies that 
		\begin{equation}\label{EQ4-6}
			\int_{\mathbb{R}^n} \varphi^2\left|\nabla u\right|^2\,dx \le \int_{\mathbb{R}^n} \varphi^4\left|\Delta^A u\right|^2\,dx+\frac{4c_2^2}{c_1^2} \int_{\mathbb{R}^n} \left|\nabla \varphi\right|^2u^2\,dx+\frac{1}{c_1^2} \int_{B_{\alpha r}} u^2\,dx.
		\end{equation}
		If we take $R_0=\max\left\{R,\sqrt{2\,D_2}, 1\right\}>0$, substitute \eqref{EQ4-6} into \eqref{EQ4-5} and use the definition of $\varphi$, then for $r >R_0$, there is a constant $C>0$ depending only on $n$, $\alpha$, $c_1$, $c_2$, $c_3$, and $c_4$ such that \eqref{RPI} holds. Hence the finite dimensional result follows.
	\end{proof}

\vskip 1cm


\providecommand{\bysame}{\leavevmode\hbox to3em{\hrulefill}\thinspace}
\providecommand{\MR}{\relax\ifhmode\unskip\space\fi MR }
\providecommand{\MRhref}[2]{%
  \href{http://www.ams.org/mathscinet-getitem?mr=#1}{#2}
}
\providecommand{\href}[2]{#2}

\end{document}